\documentclass[english, 10pt]{amsart}
\usepackage{amsmath}
\usepackage{graphicx}
\usepackage{amssymb}
\usepackage[colorlinks=true,linkcolor=blue,citecolor=blue,urlcolor=blue]{hyperref}  

\newcommand{\al}{\alpha}

\newcommand{\la}{\lambda}

\newcommand{\auskommentieren}[1]{}

\newcommand{\beq}{\begin{equation}}
\newcommand{\eeq}{\end{equation}}
\newcommand{\bea}{\begin{equation}\begin{aligned}}
\newcommand{\eea}{\end{aligned}\end{equation}}

\newtheorem{theorem}{Theorem}[section]

\newtheorem{lem}[theorem]{Lemma}

\theoremstyle{definition}

\newtheorem{rem}[theorem]{Remark}

\usepackage{array} 
\usepackage{amssymb}
\usepackage{latexsym}
\usepackage{epic}
\usepackage{curves}
\usepackage{fancyhdr}
\usepackage{xspace}

\numberwithin{equation}{section}

\begin{document}
\title{A note on expansion of convex plane curves via inverse curvature flow}
\author{Heiko Kr\"oner}
\date{June 15, 2014}
\maketitle
\begin{abstract}
Recently Andrews and Bryan \cite{AB} discovered a comparison function which allows them to shorten the classical proof of the well-known fact that the curve shortening flow shrinks embedded closed curves in the plane to a round point. Using this comparison function they estimate the length of any chord from below in terms of the arc length between its endpoints and elapsed time. They apply this estimate to short segments and deduce directly that the maximum curvature decays exponentially to the curvature of a circle with the same length. 

We consider the expansion of convex curves under inverse (mean) curvature flow and show that the above comparison function also works in this case to obtain a new proof of the fact that the flow exists for all times and becomes round in shape, i.e. converges smoothly to the unit circle after an appropriate rescaling.
\end{abstract}
\footnotetext{
\textsc{Weierstra\ss -Institut, 
Mohrenstrasse 39,
10117 Berlin,
Germany}\\
\textit {E-mail:} \href{mailto:kroener@wias-berlin.de}{kroener@wias-berlin.de} \\
\textit {Url:} \href{http://na.uni-tuebingen.de/~kroener/}{http://na.uni-tuebingen.de/$\sim$kroener/}}

\section{Main result}
In \cite{CT} the motion of a smooth, strictly convex, embedded closed curve in $\mathbb{R}^2$ expanding in the direction of its outward normal vector with speed given by an arbitrary positive increasing function $G$ of its principal radius of curvature is considered. It is shown that there exists a unique one-parameter family of smooth, strictly convex curves satisfying the above equation, which expand to infinity. The shapes of the curves become round asymptotically in the sense that if one rescales the equation appropriately, the support functions of the rescaled curves converge uniformly to the constant 1 in the $C^2$-norm, i.e. the rescaled curves converge to the unit circle. Under additional hypotheses on the function $G$ which are satisfied in  case $G(x)=x$ the convergence is in the $C^{\infty}$-norm. 

Our aim is to present a new proof of this result in the case $G(x)=x$, i.e. the inverse (mean) curvature flow (\ref{0}), which uses the distance comparison principle from \cite{AB} and obtains a curvature bound directly.

This distance comparison principle \cite{AB} is a refinement of Huisken's distance comparison principle \cite{H2}. The latter together with known classification of singularities is used in \cite{H2} to show that the curve shortening flow shrinks every embedded curve in the plane to a round point. This was originally proved by Grayson \cite{Grayson} using different ideas. 

The refinement \cite{AB} obtains a stronger control on the chord distances sufficient to imply a curvature bound. After rescaling the evolving curves to have length $2 \pi$ the curvature bound implies that the maximum curvature approaches 1 at a sharp rate. This gives a self-contained proof of Grayson's theorem which does not require the monotonicity formula or the classification of singularities.

We mention some general references. For results concerning the curve shortening flow we refer to \cite{Gage},\cite{GH} and \cite{Grayson} and for results concerning inverse curvature flows of curves to \cite{CT} and \cite{A}. For contracting and expanding flows of hypersurfaces in $\mathbb{R}^n$, $n \ge 3$, see e.g. \cite{G1}, \cite{H}, \cite{U}, \cite{G2}, \cite{S2} and \cite{SS}. For a survey of so-called two point functions (e.g. Huisken's distance comparison function) and their application in geometry we refer to \cite{B}.

Our paper is organized as follows. In this section we state and prove the key estimate, cf. Theorem \ref{4}. In Section \ref{section2} we obtain a curvature bound for the evolving rescaled curves. In section \ref{section3} we prove higher order estimates for the curvature, longtime existence of the flow and convergence of the rescaled curves to the unit sphere.
 
We consider the inverse (mean) curvature flow
\beq \label{0}
\tilde F :  S^1 \times [0, T) \rightarrow \mathbb{R}^2, \quad \frac{\partial \tilde F}{\partial \tau} = \frac{1}{\tilde \kappa}\nu, 
\eeq
with strictly convex, embedded closed initial curve $\tilde F(\cdot, 0)$,
where $\nu$ is the outer unit normal and $\tilde \kappa$ the curvature with respect to $\nu$. 

Let $\tilde F$ be a solution of (\ref{0}) where $0<T\le \infty$. We normalize the curves to have total length $2 \pi$. 
 Therefore we define $F: S^1 \times [0, T) \rightarrow \mathbb{R}^2$ by
\beq
F(p, t) = \frac{2 \pi}{L[\tilde F(\cdot, t)]}\tilde F(p, t).
\eeq
Then $L[F(\cdot, t)]=2\pi$ for every $t$, and $F$ evolves according to the normalised equation
\beq \label{1}
\frac{\partial F}{\partial t} = -F +\frac{1}{\kappa} \nu
\eeq
where $\kappa$ denotes the curvature of the normalised curve $F$.

We denote the chord length by $d(p,q,t)=|F(q,t)-F(p,t)|$ and the arc length along the curve $F(\cdot, t)$ by $l(p,q,t)$ for $p,q \in S^1$ and $0 \le t < T$. Our main result is the following

\begin{theorem} \label{4}
Let $F: S^1\times [0, T)\rightarrow \mathbb{R}^2$ be a smooth, convex embedded solution of the normalised curve-shortening flow (\ref{1}) with fixed total length $2\pi$. Then there exists $\bar t \in \mathbb{R}$ such that for every $p$ and $q$ in $S^1$ and every $t \in [0, T)$
\beq \label{2}
d(p,q,t) \ge f(l(p,q,t),t-\bar t),
\eeq 
where $f$ is defined by $f(x,t)= 2 e^t \arctan \left(e^{-t}\sin\left(\frac{x}{2}\right)\right)$ for $t \in \mathbb{R}$ and $x \in [0, 2\pi]$.
\end{theorem}

The remaining part of this section deals with the proof of Theorem \ref{4}. Thereby we adapt the proof of \cite[Theorem 1]{AB} where the corresponding result for the normalised curve shortening flow is proved. 

The first step to prove Theorem \ref{4} is to show that (\ref{2}) holds for $t=0$ if $\bar t$ is sufficiently large. This is a property for an embedded convex closed curve (and does not depend on the curvature flow under consideration) and follows from \cite[Theorem 1]{AB}. 

To show the result for positive times we use a maximum principle argument. We define $Z: S^1 \times S^1\times [0, T)$ by
\beq
Z(p,q,t) = d(p,q,t)-f(l(p,q,t), t-\bar t).
\eeq
$Z$ is continuous on $S^1 \times S^1 \times [0, T)$ and smooth where $p \neq q$. We prove by contradiction that
\beq
Z_{\epsilon}=Z + \epsilon 
\eeq
remains positive on $S^1\times S^1\times[0, T)$ for any $\epsilon >0$.
There holds 
\beq
Z_{\epsilon}\ge \epsilon >0
\eeq
at $t=0$ and on the diagonal $\{(p,p): p \in S^1\}$, so if $Z_{\epsilon}$ does not remain positive then there exists $t_0 \in (0, T)$ and $p_0 \neq q_0$ in $S^1$ such that
\beq
Z_{\epsilon}(p_0. q_0, t_0) = 0 = \inf\{Z_{\epsilon}(p,q,t): p,q \in S^1, 0 \le t \le t_0\}.
\eeq
Hence at $(p_0, q_0, t_0)$ we have
\beq
Z = -\epsilon, \quad \frac{\partial Z}{\partial t} = \frac{\partial Z_{\epsilon}}{\partial t} \le 0,
\eeq
the first spatial derivative of $Z$ vanishes and the second is non-negative.
A calculation as in the proof of \cite[Theorem 1]{AB} now shows that we have in $(p_0, q_0, t_0)$
\beq \label{5}
\left<T_{p_0}, w\right> = \left<T_{q_0}, w\right>=f', \quad T_{p_0} \neq T_{q_0}
\eeq
and
\beq
0 \le \left<w, \kappa_{p_0}\nu_{p_0}-\kappa_{q_0}\nu_{q_0}\right> - 4 f''
\eeq
where $T_{p_0}= \frac{F(p_0, t_0)}{ds}$, $s$ arc length, 
\beq
w(p,q,t)= \frac{F(q,t)-F(p,t)}{d(p,q,t)}
\eeq
and the prime $'$ denotes the derivative of $f=f(x,t)$ with respect to $x$.

It follows that
\beq
0> \left<w, \nu_{p_0}\right> = - \left<w, \nu_{q_0}\right>.
\eeq
Under the normalised equation (\ref{1}), $d$ and $l$ evolve according to
\bea
\frac{\partial  d}{\partial t} =& \frac{1}{d}\left<\frac{1}{\kappa_{p_0}}\nu_{p_0} - F_{p_0}-\frac{1}{\kappa_{q_0}}\nu_{q_0}+F_{q_0}, F_{p_0}-F_{q_0}\right> \\
=& \left<w, -\frac{1}{\kappa_{p_0}}\nu_{p_0}+\frac{1}{\kappa_{q_0}}\nu_{q_0}\right> -d,\\
\frac{\partial  l}{\partial t} =& 0 
\eea
From these equations we obtain an expression for the time derivative of $Z$
\bea \label{10}
0 \ge& \frac{\partial Z}{\partial t} \\
=& \frac{\partial d}{\partial t}-\frac{\partial f}{\partial t} \\
=& -(\frac{1}{\kappa_{p_0}}+\frac{1}{\kappa_{q_0}})\left<w, \nu_{p_0}\right>-d
-\frac{\partial f}{\partial t} \\
\ge& -\frac{\left<w, \nu_{p_0}\right>^2}{f^{''}}-d
-\frac{\partial f}{\partial t} \\
=& -\frac{\left<w, \nu_{p_0}\right>^2}{f^{''}} -f+\epsilon -\frac{\partial f}{\partial t} \\
=& \frac{{f'}^2-1}{f^{''}} -f+\epsilon -\frac{\partial f}{\partial t}.
\eea
We conclude that
\beq \label{6}
-\epsilon  \ge Lf(l(p_0, q_0, t_0), t_0-\bar t)
\eeq
where
\bea
Lf(x,t) = \frac{{f'}^2-1}{f^{''}} -f -\frac{\partial f}{\partial t}, \quad x \in (0, \pi],   t\in \mathbb{R}.
\eea
The following calculation shows that 
\beq
Lf \ge 0,
\eeq
which is a contradiction.
We have
\bea
f(x,t) =& 2e^t \arctan (e^{-t}\sin(x/2)) \\
f^{'}(x,t) = & \frac{\cos(x/2)}{1+e^{-2t}\sin^2(x/2)} \\
f^{''}(x,t) =& -\frac{1/2\sin(x/2)}{1+e^{-2t}\sin^2(x/2)}-\frac{e^{-2t}\cos^2(x/2)\sin(x/2)}{(1+e^{-2t}\sin^2(x/2))^2} \\
\frac{\partial f}{\partial t}=& 2e^tg(e^{-t}\sin(x/2)), \quad g(z) = \arctan z - \frac{z}{1+z^2}
\eea
and therefore
\bea
Lf(x,t) =& \frac{2 \sin(x/2) \{1+2e^{-2t}+e^{-4t}\sin^2(x/2)\}}{1+e^{-2t}\sin^2(x/2)+2\cos^2(x/2)e^{-2t}}\\
& -4e^t \arctan(e^{-t} \sin(x/2)) + \frac{2\sin(x/2)}{1+e^{-2t}\sin^2(x/2)}.
\eea
It follows that $\lim_{x\rightarrow 0, x>0} Lf(x,t)=0$ and we will show that $D(Lf)(x,t)\ge 0$ for $x \in (0, \pi]$ and $t\in \mathbb{R}$. We have
\bea
D(Lf)(x,t) =& - \frac{\cos(x/2)}{1+e^{-2t}\sin^2(x/2)}-\frac{2e^{-2t}\sin^2(x/2)\cos(x/2)}{(1+e^{-2t}\sin^2(x/2))^2} \\
&+ \frac{\cos(x/2)\{1+2e^{-2t}+3e^{-4t}\sin^2(x/2)\}}{1+2e^{-2t}-\sin^2(x/2)e^{-2t}} \\
&+ \frac{2 e^{-2t}\sin^2(x/2) \cos(x/2)\{1+2e^{-2t}+e^{-4t}\sin^2(x/2)\}}{(1+2e^{-2t}-e^{-2t}\sin^2(x/2))^2}.
\eea
If $x=\pi$ the claim is obvious, otherwise we divide this expression by $\cos(x/2)$, write $\alpha:= e^{-2t}$, $z=\sin(x/2)$ and get
\bea
-\frac{1}{1+\al  z^2}-\frac{2 \al z^2}{(1+\al z^2)^2} +&  \frac{1+2\al+3\al^2z^2}{1+2\al-\al z^2}+\frac{2z^2\al(1+2\al+\al^2z^2)}{(1+2\al-\al z^2)^2}\\
=:\frac{A}{(1+\al z^2)^2(1+2\al-\al z^2)^2}.
\eea
It suffices to show that $A\ge 0$ for what we calculate $A$ in detail and arrange terms by the powers of $z$. There holds
\bea
A =& (2\al+5\al^2+2\al^3)z^2 + (8 \al^2+25\al^3+16\al^4) z^4 \\
&+(-2\al^3+3\al^4+6\al^5)z^6-\al^5z^8 \\
\ge& 6 \al^5z^6(1-\frac{z^2}{6})+25 \al^3z^4(1-\frac{2}{25}z^2) \\
\ge& 0
\eea
in view of the definition of $z$.

\section{Curvature bounds} \label{section2}

There holds the following estimate for $\kappa$ from above.
\begin{theorem} \label{12}
With $\bar t$ as in Theorem \ref{4} we have
\beq \label{20}
\kappa(p,t)^2 \le 1+ 2 e^{-2(t-\bar t)}
\eeq
for  $p \in S^1$, $t\in [0, T)$.
\end{theorem}
\begin{proof}
See \cite[Theorem 3]{AB}.
\end{proof}

We remark that the difference of the curvatures of the incircle and the circumcircle of the evolving curves is estimated from above by $c e^{-t}$, $c>0$ a constant, in view of Theorem \ref{12} and the Bonnesen inequality, cf. \cite{O}, which estimates this difference from above by an isoperimetric deficit between length and enclosed volume.

We state some evolution equations (in arc length coordinates).

\begin{lem} \label{9}
For the normalised flow (\ref{1}) there holds 
\bea
\dot \nu =& - D\left(\frac{1}{\kappa}\right)DF \\
\dot \kappa -\frac{1}{\kappa^2}\Delta \kappa=& -2 \frac{(D\kappa)^2}{\kappa^3}.
\eea
\end{lem}
\begin{proof}
(i) We have 
\bea
0 =& \left<\dot \nu, DF\right> \\
=& -\left<\nu, D\dot F\right> \\
=& -\left<\nu, -DF+ D\left(\frac{1}{\kappa}\right)\nu\right> \\
=& -D\left(\frac{1}{\kappa}\right).
\eea
(ii) We use that
\bea
\frac{\partial}{\partial t}(D\nu)=& \dot\kappa DF + \kappa D\dot F 
\eea
is equal to
\beq
D \dot \nu = -\Delta \left(\frac{1}{\kappa}\right)DF + D\left(\frac{1}{\kappa}\right)\kappa \nu
\eeq
and calculate the scalar product of the resulting equation with $DF$ yielding
\beq
\dot  \kappa  = \Delta \left(-\frac{1}{\kappa}\right)
\eeq
and the claim follows.
\end{proof}

\begin{lem} \label{14}
There holds
\beq \label{7}
0< \inf \kappa(\cdot, 0) \le \kappa \le \sup \kappa( \cdot, 0)
\eeq
during the evolution of the normalised flow.
\end{lem}
\begin{proof}
Use the maximum principle and Lemma \ref{9}. 
\end{proof}

We derive an evolution equation and then a decay for $D \kappa$ from Lemma \ref{9}.
\begin{lem}
There holds
\bea \label{22}
\frac{d}{dt}(D \kappa)-\frac{1}{\kappa^2}\Delta (D\kappa) = \frac{6}{\kappa^4}(D\kappa)^3-\frac{6}{\kappa^3}D\kappa \Delta \kappa.
\eea
\end{lem}
\begin{proof}
Clear.
\end{proof}
\begin{lem} \label{40}
There holds
\beq
|D \kappa| \le c \min\{1, \frac{1}{\sqrt{t}}\}
\eeq
for $t \in (0, T)$.
\end{lem}
\begin{proof}
We define
\beq \label{20}
w = \kappa^{\la}D \kappa
\eeq
where $\la$ will be chosen later. We have
\bea \label{16}
\dot w - \frac{1}{\kappa^2}\Delta w  =& 
\la \kappa^{\la-1}D \kappa (\dot \kappa - \frac{1}{\kappa^2}\Delta \kappa) + k^{\la}(\frac{d}{dt}(D\kappa)-\frac{1}{\kappa^2}\Delta (D\kappa)) \\
& - \la(\la-1)\kappa^{\la-4}(D\kappa)^3-2 \la \kappa^{\la-3}D \kappa D^2 \kappa \\
=& (-\la+6-\la^2)\kappa^{\la-4}(D\kappa)^3 + (-6-2\la)\kappa^{\la-3}D\kappa D^2\kappa.
\eea
We set 
\beq
\varphi (t ) = \sup w(\cdot, t) = w(x_t, t)\ge 0, \quad x_t \in S^1,
\eeq
and get using $Dw(x_t, t)=0$ that  for a.e. $t \in [0, T)$
\bea
\dot \varphi =& \frac{1}{\kappa^2}\Delta w+(\la^2+5\la+6)\kappa^{4-\la}(D\kappa)^3 \\
\le& -c_0\varphi^3,
\eea
for $\la=-2.5$ where $c_0>0$ is a constant; here we used Lemma \ref{14}. This implies
\beq
D \kappa \le \frac{c}{\sqrt{t}}.
\eeq
Defining $\varphi(t) = \inf w(\cdot, t)\le 0$ and choosing $\la=-2.5$ we obtain
\beq
\dot \varphi \ge -c_1 \varphi^3,
\eeq
where $c_1>0$ is a constant and the claim follows.
\end{proof}

\section{Higher order estimates and longtime existence}  \label{section3}

The previous section implies that the curvature of the rescaled flow converges to 1 if the flow exists for all times. 
For reasons of completeness we present here a proof of the longtime existence of the flow and that the rescaled curves converge exponentially to the unit circle.

We differentiate equation (\ref{22}) and obtain
\begin{lem} 
\bea \label{24}
\frac{d}{dt}(D^2\kappa)-\frac{1}{\kappa^2}\Delta(D^2\kappa) =& -\frac{8}{\kappa^3}D\kappa D^3\kappa-\frac{24}{\kappa^5}(D\kappa)^4\\
&+\frac{36}{\kappa^4}(D\kappa)^2\Delta \kappa
-\frac{6}{\kappa^3}(\Delta \kappa)^2.
\eea
\end{lem}
\begin{rem} \label{23}
Defining $\varphi(t)=\sup D^2\kappa(\cdot, t)$ we see that there are constants $c_1, c_2>0$ so that
\beq
\varphi(t) \ge c_1 \min\{1,\frac{1}{t}\} \Rightarrow \dot \varphi(t) \le -c_2 \varphi^2
\eeq
for a.e. $t\in[0, T)$. Hence 
\beq
D^2\kappa \le c \min\{1,\frac{1}{t}\}
\eeq
for $t \in (0, T)$.
\end{rem}

We define $w$ according to (\ref{20}) where $\la$ will be chosen later. From Remark \ref{23} we know that $Dw$ is bounded from above (this property does not depend on $\la$). 
\begin{lem}
We have
\bea \label{33}
\frac{d}{dt}(Dw) - \frac{1}{\kappa^2}\Delta (Dw) =& -\frac{2}{\kappa^3}D\kappa \Delta w \\
& (-\la^2-\la+6)(\la-4)\kappa^{\la-5}(D\kappa)^4 \\
&+ (-5\la^2-3\la+36)\kappa^{\la-4}(D\kappa)^2D^2\kappa\\
&- (2\la+6)\kappa^{\la-3}(D^2\kappa)^2\\
&-(2\la+6)\kappa^{\la-3}D\kappa D^3\kappa.
\eea
\end{lem}

\begin{lem} \label{41}
There holds
\beq 
|D^2\kappa| \le c \min\{1,\frac{1}{t}\}
\eeq
for $t \in (0, T)$.
\end{lem}
\begin{proof}
We choose $\la < -3$ and set $\varphi(t) := \inf Dw(\cdot, t)=Dw(x_t, t)$, $x_t\in S^1$ suitable.  Assume 
\beq
\varphi(t) \le -c_0\min\{1, \frac{1}{t}\}
\eeq
where  $c_0>0$ is a sufficiently large constant then $D^2\kappa <0$ and we have for a.e. $t \in (0, T)$ in view of equation (\ref{33}) and $D^2w(x_t,t)=0$ that 
\bea
\dot \varphi(t) \ge c_1 \varphi(t)^2,
\eea
with a constant $c_1>0$.
\end{proof}

We will use interpolation to prove that $|D\kappa|$ decays exponentially with respect to the evolution parameter $t$.
Since $\int \kappa ds =L=2\pi$ we have
\bea \label{13}
\int (\kappa-1)^2ds =& \int \kappa^2ds-2 \int \kappa ds + L\\
=& \int(\kappa^2-1)ds \\
\le& 2 e^{-2(t-\bar t)}.
\eea
There hold the following Gagliardo-Nirenberg inequalities, cf. \cite[Theorem 3]{A2},
\bea \label{25}
\|D^i\kappa\|_{\infty} \le c(m,i)\|D^m\kappa\|_{\infty}^{\frac{2i+1}{2m+1}}\|\kappa-1\|_2^{\frac{2(m-i)}{2m+1}}.
\eea
Hence using our previous estimates we get
\bea
\|D\kappa\|_{\infty} \le& c(2,1) \| D^2 \kappa \|_{\infty}^{\frac{3}{5}} \|\kappa-1\|_2^{\frac{2}{5}}\\
\le& c(2,1, \bar t) \min\{1, t^{-\frac{3}{5}}\}e^{-\frac{2}{5}t}.
\eea
Differentiating equation (\ref{24}) and denoting expressions of order $O(D^m\kappa)$ by $O_m$, $m \in \mathbb{N}$, we have
\bea \label{42}
\frac{d}{dt}(D^3\kappa)- \frac{1}{\kappa^2}\Delta(D^3\kappa) = (O_1^2+O_2)D^3\kappa + O_1D^4\kappa + O_1^5+O_1O_2.
\eea
Using Lemma \ref{40}, Lemma \ref{41} and the maximum principle we get 
\beq
|D^3 \kappa| \le c(1+t)
\eeq
and via (\ref{25}) that $|D^2\kappa| \le c e^{-c_0t}$ where $c_0>0$ and hence using (\ref{42}) again
\beq
|D^3\kappa| \le c.
\eeq
Induction using (\ref{25}) implies that for every $\epsilon >0$ and $m \ge 1$
\beq
\|D^m\kappa\|_{\infty} \le c(m, \bar t, \epsilon) e^{-(1-\epsilon) t}.
\eeq

The length $L(t)=L[\tilde F(\cdot, t)]$ of the  un-normalised flow satisfies $\dot L = L$ so that
\beq
L(t) = L(0) e^t
\eeq 
and
\beq
\tilde \kappa= \frac{2\pi}{L(0)}e^{-t}\kappa \quad \wedge \quad \tilde F = \frac{L(0)}{2\pi} e^tF.
\eeq

Suppose that $T< \infty$ is maximal so that a solution of (\ref{0}) exists, then in view of the estimates from the previous sections $|D^m\tilde F(\cdot, t)|$ is uniformly bounded for $t \in [0, T)$ and fixed (but arbitrary) $m \in \mathbb{N}$  which is impossible.
Hence the flow exists for all times and $\kappa \rightarrow 1$ in $C^m$ exponentially for every $m \in \mathbb{N}$.

\bigskip
Acknowledgement:
During the preparation of this article the author benefited from a Weierstrass Postdoctoral 
Fellowship of the Weierstra\ss -Institut Berlin. We acknowledge this funding and thank for the hospitality of the institute.



\end{document}